\documentclass[11pt]{article}

\usepackage[T1]{fontenc}
\usepackage[utf8]{inputenc}
\usepackage{authblk}

\usepackage{amssymb,amsmath,amsthm}
\usepackage[colorlinks=true,urlcolor=blue,
citecolor=red,linkcolor=blue,linktocpage,pdfpagelabels,bookmarksnumbered,bookmarksopen]{hyperref}
\usepackage{graphicx}
\usepackage[english]{babel}
\usepackage{epsfig}
\usepackage{graphics}
\usepackage{amsthm}
\theoremstyle{plain}

\newtheorem{remark}{Remark}
\newtheorem{lemma}{Lemma}[section]
\newtheorem{theorem}[lemma]{Theorem}

\theoremstyle{definition}

\newtheorem*{conj*}{Conjecture}

\newtheorem{prop}[lemma]{Proposition}

\newcommand{\R}{{\mathbb R}}

\newcommand{\iy}{\infty}

\newcommand{\DR}{{\cal D}}

\newcommand{\bn}{\bigbreak\noindent}

\def\bean#1\eean{\begin{eqnarray*}#1\end{eqnarray*}}

\begin{document}

\title{A Liouville theorem for the $p$-Laplacian and related questions}
\author[1]{Alberto Farina\thanks{alberto.farina@u-picardie.fr}}
\author[2]{Carlo Mercuri\thanks{c.mercuri@swansea.ac.uk}}
\author[3]{Michel Willem\thanks{michel.willem@uclouvain.be}}

\affil[1]{LAMFA, CNRS UMR 7352, Facult\'e des Sciences, Universit\'e de Picardie Jules Verne, 33, Rue Saint-Leu, 80039 Amiens Cedex 1, France}
\affil[2]{Department of Mathematics, Computational Foundry, Swansea University, Fabian Way, Swansea, SA1~8EN, UK}
\affil[3]{D\'epartement de Math\'ematiques, Universit\'e catholique de Louvain, Chemin du Cyclotron, 2, B-1348, Louvain-la-Neuve, Belgium}
\renewcommand\Authands{ and }

\date{}


\maketitle

\begin{abstract}
We prove several classification results for $p$-Laplacian problems on bounded and unbounded domains, and deal with qualitative properties of sign-changing solutions to $p$-Laplacian equations on $\R^N$ involving critical nonlinearities. Moreover, on radial domains we characterise the compactness of possibly sign-changing Palais-Smale sequences. 
\tableofcontents

\end{abstract}
 

\section{Introduction}
Throughout the paper we use the following notation:
$$
\begin{array}{l}
\Delta_pu:={\rm div}(|\nabla u|^{p-2}\nabla u), \quad 1<p<\infty,\\
\\
p^*:=N p/(N-p),\quad 1<p<N,\\
\\
\mathcal D^{1,p}(\R^N)=\{u\in L^{p^*}(\R^N):\nabla u\in L^p(\R^N;\R^N)\},\\
\\
||u|| =||\nabla u||_{L^p(\R^N)},\\
\\
H=\R^N_+=\{x\in\R^N:x_N>0\}.
\end{array}
$$
For a smooth possibly unbounded domain $\mathcal O$ we denote by $\mathcal D_0^{1,p}(\mathcal O)$ the closure of $\mathcal D (\mathcal O)$ in $\mathcal D^{1,p}(\R^N)$. When $\mathcal O$ is bounded we set $\mathcal D_0^{1,p}(\mathcal O)=W_0^{1,p}(\mathcal O).$\\
Let $1<p<N,$ $\Omega$ be a smooth bounded domain of $\mathbb R^N$ and $a\in L^{N/p}(\Omega).$ 
We define  on $W^{1,p}_0(\Omega)$ $$\phi(u)=\int_{\Omega}\Big(\frac{|\nabla u|^{p}}{p}+a(x)\frac{|u|^p}{p}-\frac{|u|^{p^*}}{p^*}\Big)\textrm{d}x,$$ and on $\mathcal D^{1,p}(\R^N)$

 $$\phi_\infty(u)=\int_{\R^N}\Big(\frac{|\nabla u|^{p}}{p}-\frac{|u|^{p^*}}{p^*}\Big)\textrm{d}x.
 $$
Recall that $$(\phi'(u),h)= \int_{\Omega}[|\nabla u|^{p-2}\nabla u \cdot \nabla h+a(x)|u|^{p-2}u \, h-|u|^{p^*-2} u \,h]\textrm{d}x,$$

$$(\phi_{\infty}'(u),h)= \int _{\R^N}[|\nabla u|^{p-2}\nabla u\cdot \nabla h- |u|^{p^*-2}u\,h]\textrm{d}x.$$
From \cite{MW} a blow-up theory for the Palais-Smale sequences $\{u_n\}_n \subset W^{1,p}_{0} (\Omega)$ for $\phi$ is available when $\{u_n\}_n$ is a bounded sequence `nearby' the positive cone of $W^{1,p}_{0} (\Omega).$ 
We assume here that

 $$\phi(u_n)\rightarrow c \quad \quad \phi'(u_n)\rightarrow 0 \quad \textrm{in} \,\,W^{-1,p'}(\Omega) $$ and
 $$\|(u_n)_-\|_{L^{p^*}(\Omega)}\rightarrow 0, \quad n\rightarrow \infty.
 $$
 Then, a $p$-Laplacian generalisation of Struwe's global compactness result, see e.g. \cite{MaWi, Struwe, Will} holds. In fact by \cite{MW}, passing if necessary to a subsequence, there exists a possibly nontrivial solution $v_0\in W^{1,p}_{0}(\Omega)$ to
\begin{eqnarray*}
-\Delta_p u +a(x) u^{p-1}=  u^{p^*-1} &\textrm{in}& \Omega, \\
u \geq 0 & \textrm{in} & \Omega,
\end{eqnarray*}
$k$ possibly nontrivial solutions $\{v_1,...,v_k\}\subset \mathcal D^{1,p}(\R^N)$ to
\begin{eqnarray*}
-\Delta_p u = u^{p^*-1} &\textrm{in}& \R^N, \\
u \geq 0 & \textrm{in} & \R^N,
\end{eqnarray*}
 and $k$ sequences $\{y^i_n\}_n \subset \Omega$ and $\{\lambda^i_n\}_n \subset \R_+,$  such that

$$\frac{1}{\lambda^i_n}\, \textrm{dist} \, (y^i_n,\partial\Omega)\rightarrow \infty , \,\quad n\rightarrow \infty, $$

$$\|u_n-v_0-\sum^k_{i=1}(\lambda^i_n)^{(p-N)/p}v_i ((\cdot-y^i_n)/\lambda^i_n)\|\rightarrow 0, \quad n\rightarrow \infty,$$

$$\|u_n\|^p\rightarrow \sum^k_{i=0}\|v_i\|^p, \quad n\rightarrow \infty,$$

$$\phi(v_0)+\sum^k_{i=1}\phi_\infty (v_i)=c.$$
Recent symmetry results of Sciunzi \cite{S}  (see also V\'etois \cite{Vet}) together with the uniqueness of the radial positive solutions to $-\Delta_p u = u^{p^*-1}$ on $\R^N$  (see Guedda-Veron \cite{GV2}), allow to prove that the limiting profiles $\{v_1,...,v_k\}$ are given by the classical Talenti functions (\cite{T}). Among the various applications of this result, it is worth mentioning \cite{MSS}, which extends to the case $p\neq 2$ the classical result of Coron \cite{coron}.\\
When considering arbitrary sign-changing Palais-Smale sequences the scenario is much richer. In fact, one may have $$\liminf_{n\rightarrow \infty}\frac{1}{\lambda^i_n}\, \textrm{dist} \, (y^i_n,\partial\Omega)< \infty $$ and as a consequence some of the limiting functions $v_i$ may live on a half-space. Ruling out that this situation may occur would yield a complete generalisation of Struwe's result for the $p$-Laplace operator. More precisely, one may conjecture the following Liouville-type theorem
\begin{conj*}\label{Halfspace}
{\it Let $u\in \DR^{1,p}_0(\R_+^N)$ be a weak solution of the equation

\begin{equation}\label{imp}
-\Delta_p u =   |u|^{p^*-2}u \quad \textrm{in}\quad \R_+^N.
\end{equation}  Then $u\equiv 0.$}
\end{conj*}
In \cite{MW} it has been shown that this conjecture is true under the additional assumption $u\geq 0.$ In a more delicate regularity setting, following \cite{EL} the proof in \cite{MW} consists in showing that the normal derivative of a nontrivial solution vanishes along the boundary of $\R_+^N.$ Therefore $u$ extends by zero to a solution on $\R^N,$ contradicting the strong maximum principle \cite{V,PuSe}.  However, when dealing with the $p$-Laplacian operator, a unique continuation principle seems to be a major open question. \\
\subsection{Main results}Among the main results of the present paper we have an a priori quantitative bound on the number of nodal regions for the solutions  to (\ref{imp}). More precisely we have the following general result for possibly unbounded domains.
Hereafter we refer to the Sobolev constant as defined as
\begin{equation}\label{best}
S=S(N,p):=\inf\Big\{\int_{\R^N}|\nabla u|^p \textrm{d}x,\, u\in {\mathcal D}^{1,p}(\R^N)\, : \, \int_{\R^N}|u|^{p^*} \textrm{d}x =1\Big\},
\end{equation}
achieved on functions 

$$U_{\lambda, x_0}:=  \Big[\frac{\lambda^{\frac{1}{p-1}}N^{\frac{1}{p}}(\frac{N-p}{p-1})^{\frac{p-1}{p}}} {\lambda^{\frac{p}{p-1}}+|\cdot-x_0|^{\frac{p}{p-1} }}\Big]^{\frac{N-p}{p}} , \quad\lambda\in\R_+, x_0\in \R^N,
$$
see \cite{T} .

\begin{theorem}\label{prep}
Let $1<p<N,$ let $\mathcal O$ be a smooth domain of $\R^N$ and $u\in \DR_0^{1,p}(\mathcal O)$ be a solution to the equation

\begin{equation}\label{eqspace}
-\Delta_p u =   |u|^{p^*-2}u \quad \textrm{in}\quad \mathcal D'(\mathcal O).
\end{equation}  
Then
\begin{itemize}
  \item [] $i)$ for every nodal domain $\omega$ of $u$ it holds that $$\int_\omega|\nabla u|^p=\int_\omega|u|^{p^*};$$
  \item [] $ii)$ if $u\in \DR^{1,p}_0(\mathcal O)\setminus \{0\}$ then $u$ has at most a finite number of nodal domains. More precisely let $\mathcal N_u$ be the set of nodal domains of $u$ and $\sharp \mathcal N_u$ its cardinality, it holds that
$$
   \sharp \mathcal N_u\leq S(N,p)^{-N/p}\int_{\mathcal O}|u|^{p^*}
$$
 where $S(N,p)$ is the best Sobolev constant defined in (\ref{best}). 
  \end{itemize}  
\end{theorem}
\noindent As a consequence of the above theorem we have two propositions in a radially symmetric setting, which are of independent interest. \\ 
For $R,\mu>0$ consider the radial problem
\begin{equation}\label{radial bounded}
\left\{\begin{array}{l}
-\Delta_pu=\mu |u|^{p^*-2}u\quad \mbox{~in~}B(0,R)\subset\R^N,\\
u\in W_0^{1,p}(B(0,R)).
\end{array}
\right.
\end{equation}
For $p=2$, 0 is the only solution by the unique continuation principle. When $\frac{2N}{N+2}\leq p\leq2$, 0 is the only radial solution, see \cite{MW} p. 482. 
 Following a different method, we can now improve this nonexistence result for all $p\in (1,N).$  Set $B=B(0,R).$ We have the following
 \begin{prop}\label{nonexx}
 {\it Let $1<p<N,$  and let $u\in W^{1,p}_0(B)$ be a possibly sign-changing radial weak solution to the equation (\ref{radial bounded}). Then $u\equiv 0.$} \end{prop}

\begin{prop} \label{nonexxx}
{\it Let $1<p<N$  and let $u\in \mathcal D^{1,p}(\R^N)\setminus \{0\}$ be a possibly sign-changing radial weak solution to the equation

\begin{equation}\label{wholespace}
-\Delta_p u =   |u|^{p^*-2}u \quad \textrm{in}\quad \R^N.
\end{equation}
Then, necessarily 
\begin{equation}\label{classification}
u\equiv U_{\lambda}:=\pm  \Big[\frac{\lambda^{\frac{1}{p-1}}N^{\frac{1}{p}}(\frac{N-p}{p-1})^{\frac{p-1}{p}}} {\lambda^{\frac{p}{p-1}}+|\cdot|^{\frac{p}{p-1} }}\Big]^{\frac{N-p}{p}}  
\end{equation}
for some $\lambda>0.$} 
\end{prop}
\noindent In Section \ref{radial section} we show that the above propositions yield a precise representation of the Palais-Smale sequences for radial problems, see Proposition \ref{radialvariant} and Proposition \ref{radialB}.\\

\noindent In a nonradial setting we have the following classification results by means of the Morse index, we recall its definition in Section \ref{morsection}. The assumption $p>2$ allows to have twice differentiability of the associated energy functionals. 
In the spirit of \cite{DFSV}, the following theorem states that the number of nodal regions of a solution cannot exceed its own index.
\begin{theorem}\label{morse theorem}
Let $\mathcal O$ be a smooth domain of $\R^N,$ $2<p<N,$ and let $u\in \DR_0^{1,p}(\mathcal O)$ be a solution to the equation

\begin{equation*}
-\Delta_p u =   |u|^{p^*-2}u \quad \textrm{in}\quad  \mathcal D'(\mathcal {O}).
\end{equation*}  
Then 

$$\sharp \mathcal N_u \leq i(u),$$
where $i(u)$ is the Morse index of $u.$
\end{theorem}
As a consequence of the above theorem we have the following classification results which are suitable when studying solutions with min-max methods, see e. g. \cite{Will}. \\ 
In the spirit of \cite{F} we have the following 
\begin{theorem}\label{Halfspace0}
Let $2<p<N.$ \\
1) Let $u\in \DR^{1,p}_0(\R_+^N)$ be a weak solution to the equation

\begin{equation}\label{hh}
-\Delta_p u =   |u|^{p^*-2}u \quad \textrm{in}\quad \R_+^N.
\end{equation}  Then $u\equiv 0$ if and only if $i(u)\leq 1.$ \\
2) Let $u\in \DR^{1,p}(\R^N)$ be a weak solution to  
\begin{equation}\label{hhh}
-\Delta_p u =   |u|^{p^*-2}u \quad \textrm{in}\quad \R^N.
\end{equation} 
If $i(u)\leq1,$ then either $u\equiv 0$ or for some $x_0\in \R^N$ and $\lambda>0$ and up to the sign
$$u\equiv U_{\lambda, x_0}:=  \Big[\frac{\lambda^{\frac{1}{p-1}}N^{\frac{1}{p}}(\frac{N-p}{p-1})^{\frac{p-1}{p}}} {\lambda^{\frac{p}{p-1}}+|\cdot-x_0|^{\frac{p}{p-1} }}\Big]^{\frac{N-p}{p}}  .
$$
\end{theorem}
\begin{remark}\label{Rk3} To the best of our knowledge it is not clear whether Talenti's functions have index exactly equal to 1.
\end{remark}
\noindent For bounded domains we have the following  
\begin{theorem}\label{bddstar} Let $2<p<N,$ and $\Omega$ is a smooth bounded domain of $\R^N$, starshaped about the origin, namely such that $x \cdot \nu \geq 0$ on $\partial \Omega,$ where $\nu$ 
is the exterior normal unit vector. Let $u\in W^{1,p}_0(\Omega)$ be such that
\begin{equation}\label{eqOm}
-\Delta_p u =   |u|^{p^*-2}u \quad \textrm{in}\quad \mathcal D'( \Omega).
\end{equation} 
If $i(u)\leq 1,$ then it holds that $u\equiv 0.$
\end{theorem}

\begin{remark}
In the case $p=2$ and $N\geq 3,$ this result is well-known without any restriction on the index, as a consequence of the unique continuation principle.  
\end{remark}
\begin{remark}
If $\Omega$ is an Esteban-Lions type domain (namely a generalisation of Pohozaev's starshaped domains, see \cite{poh}), we believe that Theorem \ref {bddstar} holds, see \cite{EL}.
\end{remark}

\section{The normal derivative vanishes at the boundary}\label{normal section}
In this section we show that, regardless of their sign, solutions to (\ref{imp}) have vanishing normal derivative along the
boundary. Hereafter we set $H=\mathbb R_+^N=\{x\in\R^N:x_N>0\},$ and $B_R=\{x\in \mathbb R^N\,:\,|x|<R\},$ for some $R>0.$ Moreover, we denote by $n(\cdot)$ the exterior unit normal to $\partial(H\cap B_R)$ whose $N$-th component is $n_N(\cdot).$ The $N$-th partial derivative will be denoted by $\partial_N.$
\begin{lemma} \label{lem}
Let $1<p<N$ and $u\in \mathcal D ^{1,p}_0(H)$ be a weak solution to equation (\ref {imp}).
Then $\partial_N u=0$  everywhere on $\partial H$.
\end{lemma}
\begin{proof}
The case $1 <p < 2$ had been obtained in \cite{MW} without any positivity assumption, while the case $p=2$ is known from \cite{EL}. \\
For $2<p<N$ we argue as follows. We observe that solutions of (\ref{imp}) are $C^{1,\alpha}_{\textrm{loc}}(\bar{H}),$ see e.g. \cite{Di, Peral, Tr}. \\
As in \cite{MW} we prove the following local Pohozaev's identity, in the spirit of a similar identity proved in \cite{EL} in the case $p=2:$

\begin{equation}\label{poh}
\begin{split}
\left(1-\frac{1}{p}\right)\int_{B_R\cap\partial H}|\partial_N u|^p \textrm{d}\sigma & =\int_{H\cap\partial B_R}[\partial_Nu |\nabla u|^{p-2}\nabla u\cdot n(\sigma)-
\frac{|\nabla u|^p}{p}n_N(\sigma)]\textrm{d}\sigma \\
& \quad \quad \quad +\int_{H\cap\partial B_R}\frac{|u|^{p^*}}{p^*}n_N(\sigma)\textrm{d}\sigma.
\end{split}
\end{equation}
The desired conclusion will be then achieved. Indeed, since $\nabla u\in L^p(H)$ and $u\in L^{p^*}(H)$ the right hand side is bounded by a function $M(R)$ such that for some sequence $R_k\rightarrow\infty,$ $M(R_k)\rightarrow 0.$  \\ In order to prove (\ref{poh}) we use a regularisation argument, see e.g. \cite{Di} and \cite{GV}. We point out that in \cite{GV} a Pohozaev identity for the $p$-Laplacian is available in the context of Dirichlet problems on bounded domains.\\
By antireflection with respect to $\partial H$ extend (and still denote by) $u$ to a solution on the whole $\R^N.$ Following [p. 833, \cite{Di}], we consider $u_\varepsilon$ solution to the boundary value problem
\begin{align*}
-\textrm{div}\left(\left(\varepsilon+ |\nabla u_\varepsilon|^2\right)^{\frac{p-2}{2}}\nabla u_\varepsilon\right)= |u|^{p^*-2}u \quad \quad &\textrm{in}\,\, B_{2R}, \\
u_\varepsilon = u \quad \quad  & \textrm{on} \,\,  \partial B_{2R} ;
\end{align*}

 $u_\varepsilon \in  C^2(\bar{B}_R)$ and uniformly bounded for $\varepsilon \in (0, 1]$ in $C^{1,\alpha}(\bar{B}_R).$ 
 By the Ascoli-Arzel\'a theorem for a suitable sequence $\varepsilon \rightarrow 0 ^+,$ $u_\varepsilon\rightarrow u$ and $\nabla u_\varepsilon  \rightarrow \nabla u$
uniformly on $\bar{B}_R$. \\
Consider the vector field $$v_\varepsilon:=\left(\varepsilon+ |\nabla u_\varepsilon|^2\right)^{\frac{p-2}{2}}\partial_N u_\varepsilon\nabla u_\varepsilon.$$
Since $$\textrm{div} \,v_\varepsilon=\partial_N  u_\varepsilon{\textrm{ div~}}\left(\left(\varepsilon+ |\nabla u_\varepsilon|^2\right)^{\frac{p-2}{2}}\nabla u_\varepsilon\right)+\left(\varepsilon+ |\nabla u_\varepsilon|^2\right)^{\frac{p-2}{2}}\nabla u_\varepsilon\cdot\nabla\partial_N u_\varepsilon,$$
by the divergence theorem we obtain
\begin{align*}
&\int_{B_R\cap H}\partial_N  u_\varepsilon{\textrm{ div~}}\left(\left(\varepsilon+ |\nabla u_\varepsilon|^2\right)^{\frac{p-2}{2}}\nabla u_\varepsilon\right)\textrm{d}x\\ & \quad \quad \quad =\int_{\partial(H\cap B_R)}v_\varepsilon \cdot n(\sigma)\textrm{d} \sigma -\int_{B_R\cap H}\left(\varepsilon+ |\nabla u_\varepsilon|^2\right)^{\frac{p-2}{2}}\nabla u_\varepsilon\cdot\nabla\partial_N u_\varepsilon \textrm{d}x
 \\
& \quad \quad \quad =\int_{\partial(H\cap B_R)}v_\varepsilon \cdot n(\sigma)\textrm{d} \sigma-\int_{\partial (H\cap B_R)}\frac{[\varepsilon+|\nabla u_\varepsilon|^2]^{p/2}}{p} n_N(\sigma)\textrm{d}\sigma.
\end{align*}
Moreover
$$
\int_{H\cap B_R}\partial_Nu|u|^{p^*-2} u~\textrm{d}x =\int_{\partial(H\cap B_R)}\frac{|u|^{p^*}}{p^*}n_N(\sigma)\textrm{d}\sigma =  \int_{H\cap \partial B_R} \frac{|  u|^{p^*}}{p^*} n_N(\sigma)\textrm{d}\sigma.
$$
Equating the above two expressions and passing to the limit, as $\varepsilon \rightarrow 0^+$ we obtain (\ref{poh}). This concludes the proof.
\end{proof}
\section{General facts about nodal regions and proof of Theorem \ref{prep}}
The following approximation result will be used to prove Theorem \ref{prep} and Theorem \ref{morse theorem}.
\begin{lemma}\label{approx}
Let $u\in C^{0,1}_{\textrm{loc}}(\R^N)\cap \mathcal D^{1,p}(\R^N)$ and let $\omega$ be a nodal domain of $u,$ and $u|_\omega$ its restriction to $\omega.$ Then there exists a sequence $\{u_n\}_n\subset C^{0,1}_c(\omega)$ such that 
\begin{itemize}
  \item [] $i)$ $u_n \rightarrow u|_\omega$ in $L^{p^*}(\omega)$ and everywhere in $\omega$,
  \item [] $ii)$ $\nabla u_n \rightarrow \nabla u|_\omega$ in $L^p(\omega; \R^N)$ and almost everywhere in $\omega$. 
  \end{itemize}

\end{lemma}
\begin{proof}
We can suppose that $ u>0$ in $\omega$. Let $f\in C^1(\R;\R)$ an odd function such that

\begin{equation*}\label{f}
f(t)=
\left\{
\begin{array}{ll}
0,& \mathrm{if} \,\, |t|\leq 1, \\
t, & \,\,\mathrm{if}\,\, |t|\geq 2,
\end{array}
\right.
\end{equation*}
 and define for all $n\in \mathbb N,$ $f_n(t):=\frac{1}{n}f(nt).$
We also define $v:= u|_\omega, v_n:=f_n(v).$ It is standard to see that 

\begin{itemize}
  \item []  $v_n \in C_{\textrm{loc}}^{0,1}(\omega)\cap L^{p^*}(\omega)$
  \item []  $\nabla v_n \in L^{\infty}_{\textrm{loc}}(\omega ;\R^N)\cap L^p(\omega; \R^N)$
  \item[]  $\textrm{supp} v_n \subseteq \{x\in \omega \,\,|\,\, u(x)\geq 1/n\} \subset \omega,$ 
  \end{itemize}
moreover by the dominated convergence theorem and the definition of $f_n$ we have 
 \begin{itemize}
  \item [] $a)$ $v_n \rightarrow u|_\omega$ in $L^{p^*}(\omega)$ and everywhere in $\omega$,
  \item [] $b)$ $\nabla v_n \rightarrow \nabla u|_\omega$ in $L^p(\omega; \R^N)$ and almost everywhere in $\omega$.
  \end{itemize}  
 Let now $\theta \in C^1(\R_+),$ with $\theta(t) \in [0,1]$, and such that 
 
 \begin{equation*}\label{theta}
\theta(t)=
\left\{
\begin{array}{ll}
0,& \mathrm{if} \,\, t \geq 2 \\
1, & \,\,\mathrm{if}\,\, 0\leq t \leq 1
\end{array}
\right.
\end{equation*}
and define $$\theta_n(x):=\theta \big(\frac{x}{n}\big).$$
Finally define $u_n:=\theta_n v_n.$ It is immediate to verify that $i)$ and $ii)$ hold.
\end{proof}
\noindent We are in the position to prove Theorem \ref{prep}.
\begin{proof}[Proof of Theorem \ref{prep}]
Since $u\in C^1(\overline{\mathcal O}),$ extending $u$ by zero outside $\mathcal O$ we have that $u\in C^{0,1}_{\textrm{loc}}(\R^N)\cap \mathcal D^{1,p}(\R^N).$ \\
Proof of $i). $ Pick $(u_n)$ given by Lemma \ref{approx} extending by zero outside $\omega.$ By a standard density argument for every $n\in\mathbb N$ one can test (\ref{eqspace})
with $u_n,$ namely $$\int_\omega |\nabla u|^{p-2}\nabla u\nabla u_n=\int_\omega|u|^{p^*-2}uu_n.$$ By Lemma \ref{approx}, as $n\rightarrow \infty$ $i)$ follows. \\
Proof of $ii).$ For $p\in(1,N)$ and for all $v\in \DR^{1,p}(\R^N)$ we write Sobolev's inequality as
\begin{equation}\label{Sobolev}
S(N,p)\Big(\int_{\R^N}|v|^{p^*}\Big)^{p/p^*}\leq \int_{\R^N}|\nabla v|^p.
\end{equation}
Using Sobolev's inequality with $v=u|_\omega $ extended by zero outside $\omega$
we have by $i)$
\begin{equation}\label{Sobolevv}
\int_{\omega}|\nabla u|^p = \int_{\R^N}|\nabla v|^p\geq S(N,p)\Big(\int_{\omega}|u|^{p^*}\Big)^{p/p^*}=S(N,p)\Big(\int_{\omega}|\nabla u|^{p}\Big)^{p/p^*},
\end{equation}
hence
\begin{equation}\label{Sobol}
\Big(\int_{\omega}|\nabla u|^p\Big)^{1-\frac{p}{p^*}}\geq S(N,p),
\end{equation}
namely 
\begin{equation}\label{Sobolevvv}
\int_{\omega}|\nabla u|^p\geq S(N,p)^{N/p}.
\end{equation}
It follows that $$\int_{\mathcal O}|u|^{p^*}= \int_{\mathcal O}|\nabla u|^p=\sum_{\omega\in \mathcal N_u} \int_{\omega}|\nabla u|^p\geq \sum_{\omega\in \mathcal N_u} S(N,p)^{N/p}= S(N,p)^{N/p}\, \sharp \mathcal N_u.$$ And this concludes the proof.
\end{proof}

\section{Radial problems and proof of Proposition \ref{nonexx} and Proposition \ref{nonexxx}}\label{radial section}
 
\begin{proof}[Proof of Proposition \ref{nonexx}]
It is standard to see that $u\in C^{1,\alpha}(\overline{B}),$ see \cite{Di, Li, Peral, Tr}. Suppose $ u \not \equiv 0$. By the strong maximum principle and \cite{GV}, the solution $u$ must change sign (and so it has  a zero in $B \setminus \{0\}$).  The nodal regions of $u$ are spherically symmetric, and the number of those is finite, by Theorem \ref{prep}. 
Now pick a nodal region, say $A=\{x\in B\,:\, R_1<|x|<R_2\}$ with $0<R_1<R_2\leq R.$ We can assume that $u$ solves 
\begin{equation*}
\left\{\begin{array}{lll}
&-\Delta_pu=\mu |u|^{p^*-2}u\qquad &\textrm{~in~} B(0,R_2)\\
&u>0 \qquad &\textrm{~in~}  A \\
&u=0 \qquad &\textrm{~on~} \partial A.
\end{array}
\right.
\end{equation*} 
By Pohozaev's identity, Theorem 1.1 of \cite{GV}, we have $\nabla u=\bf{0}$ on $\partial  B(0,R_2),$ and this is in contradiction with Hopf's boundary point lemma, see e.g. \cite{V}, since $u$ is positive in $A$. This concludes the proof. \end{proof}

\begin{proof}[Proof of Proposition \ref{nonexxx}]
By \cite{GV2} p.160 and the strong maximum principle, it is enough to prove that $u$ does not change sign. 
Let us assume that $u$ changes sign. The nodal regions are spherically symmetric and their number is finite. Therefore, by replacing $u$ by $-u$ if necessary,  there exists $\overline{R}>0$ large enough such that 
\begin{eqnarray}
\nonumber&u\equiv 0,\,\,&\textrm{on}\,\, {\partial B(0,\overline{R})},\\
\nonumber& u > 0,&\textrm{on}\,\, \R^N\setminus \overline{B(0,\overline{R})},
\end{eqnarray}
and so, by Proposition \ref{nonexx}
\begin{eqnarray}
\nonumber&u\equiv 0,\,\,&\textrm{on}\,\, \overline{B(0,\overline{R})},\\
\nonumber& u> 0,&\textrm{on}\,\, \R^N\setminus \overline{B(0,\overline{R})}.
\end{eqnarray}
On the other hand by continuity $\nabla u =\bf{0}$ on $\partial B(0,\overline{R}),$ and this contradicts Hopf's boundary point lemma on $\R^N\setminus \overline{B(0,\overline{R})}.$ This concludes the proof.

\end{proof}

Consider the following assumptions:
\\
\textbf{(A)} \quad \quad  $\Omega$ is the unit ball in $\R^N,$ $1<p<N, $  $a\in L^{N/p}_{\textrm{rad}}(\Omega).$ 
Assume also
$$\textbf{(B)}\quad \quad \quad\quad  \quad \quad \quad\quad \quad \quad \quad  \quad  \mathop{\inf_{u \in W^{1,p}_0(\Omega)}}_{\|\nabla u\|_{L^p}=1} \int_{\Omega}[|\nabla u|^{p}+a(x)|u|^p]\textrm{d}x>0.\quad
\quad \quad \quad \quad \quad \quad \quad \quad \quad \quad \quad\quad\quad \quad \quad \quad \quad \quad \quad\quad\quad $$
\\
Define on $W^{1,p}_0(\Omega)$ $$\phi(u)=\int_{\Omega}\Big(\frac{|\nabla u|^{p}}{p}+a(x)\frac{|u|^p}{p}-\frac{|u|^{p^*}}{p^*}\Big)\textrm{d}x,$$ and denote by $W_{0,{\rm rad}}^{1,p}(\Omega)$ (resp. $\DR_{{\rm rad}}^{1,p}(\R^N)$) the space of radial functions in $W_0^{1,p}(\Omega)$ (resp. $\DR^{1,p}(\R^N)$). We also define on $\mathcal D_{\textrm{rad}}^{1,p}(\R^N)$

$$
\tilde\phi_{\iy}(u)=\int_{\R^N}\Big(\frac{|\nabla u|^p}{p}-\frac{|u|^{p^*}}{p^*}\Big)\textrm{d}x.
$$

\begin{prop} \label{radialvariant} 
{\it Under assumptions {\bf (A)}, {\bf{(B)}} let $\{u_n\}_n$ be a sequence in $W^{1,p}_{0,{\rm rad}} (\Omega)$ such that $$\phi(u_n)\rightarrow c \quad \quad \phi'(u_n)\rightarrow 0 \quad \textrm{in} \,\,(W^{1,p}_{0,\rm rad}(\Omega))'. $$ 
\noindent Then, passing if necessary to a subsequence, there exists a possibly nontrivial solution $v_0\in W^{1,p}_{0,{\rm rad}}(\Omega)$ to

$$-\Delta_p u +a(x)|u|^{p-2}u=  |u|^{p^*-2}u,$$
and  $k$ sequences  $\{\lambda^i_n\}_n \subset \R_+,$ with  $\lambda_n^i \rightarrow 0$, $n\rightarrow \infty$, such that

$$\|u_n-v_0-\sum^k_{i=1}(\lambda^i_n)^{(p-N)/p}v_i (\cdot /\lambda^i_n)\|\rightarrow 0,$$

$$\|u_n\|^p\rightarrow \sum^k_{i=0}\|v_i\|^p,$$

$$\phi(v_0)+\sum^k_{i=1}\tilde\phi_\infty (v_i)=c,$$
where $v_i$ is either identically zero, or for some $\lambda>0$ and up to the sign, it holds that

\begin{equation*}
v_i \equiv   \Big[\frac{\lambda^{\frac{1}{p-1}}N^{\frac{1}{p}}(\frac{N-p}{p-1})^{\frac{p-1}{p}}} {\lambda^{\frac{p}{p-1}}+|\cdot |^{\frac{p}{p-1} }}\Big]^{\frac{N-p}{p}} ,
\end{equation*}
and $$\tilde\phi_{\infty}(v_i)=\frac{S^{N/p}}{N} .$$
Moreover if $a\equiv 0,$ then all weakly convergent subsequences of $\{u_n\}_n$ are weakly convergent to zero in $W^{1,p}_{0,{\rm rad}} (\Omega).$ In particular $v_0\equiv 0$ and 
hence $\phi(v_0)=0.$}

\end{prop}

\begin{proof}
Let $a\equiv0.$ Since the weak limit of $\{u_n\}_n,$ $v_0,$ solves equation (\ref{radial bounded}) with $R,\mu=1,$ then by Proposition \ref{nonexx} $v_0\equiv 0.$ 
The rest of the proof follows by Theorem 5.1 in \cite{MW} and Proposition \ref{nonexxx}. In this radial setting $\tilde\phi_\infty (v_i)$ can be computed explicitly by \cite{T}, using Proposition \ref{nonexxx}.  And this concludes the proof.
\end{proof}
\noindent Now  define on $\mathcal D^{1,p}(\R^N)$

$$
\phi(u)=\int_{\R^N}\Big(\frac{|\nabla u|^p}{p}+a(x)\frac{|u|^p}{p}-\frac{|u|^{p^*}}{p^*}\Big)\textrm{d}x
$$
and
$$\tilde\phi_{\infty}(u):=\int_{\R^N}\Big(\frac{|\nabla u|^p}{p}-\frac{|u|^{p^*}}{p^*}\Big)\textrm{d}x.$$
We assume 

\bn {\bf (C)} \quad $1<p<N$, and $a\in L^{N/p}(\R^N)$ is radial such that

$$
\inf_{u\in\mathcal D_{\rm rad}^{1,p}(\R^N)\atop{||\nabla u||_{L^p}=1}}\int_{\R^N}[|\nabla u|^p+a(x)|u|^p]\textrm{d}x >0.
$$

\begin{prop} \label{radialB}
{\it Under assumption {\bf (C)}, let $\{u_n\}_n$ be a sequence in $\mathcal D_{\rm rad} ^{1,p}(\R^N)$ such that
 $$\phi(u_n)\rightarrow c \quad \quad \phi'(u_n)\rightarrow 0 \quad \textrm{in} \,\,(\mathcal D_{\rm rad} ^{1,p}(\R^N))'. $$
 Then, passing if necessary to a subsequence, there exists a possibly nontrivial solution $v_0\in\mathcal D_{\rm rad}^{1,p}(\R^N)$ to

 $$
 -\Delta_pu+a(x)|u|^{p-2}u=|u|^{p^*-2}u
 $$
 and $k$ sequences $\{\lambda^i_n\}_n \subset \R_+,$ such that $\lambda^i_n\rightarrow 0$ or $\lambda^i_n\rightarrow\infty$ satisfying

$$\|u_n-v_0-\sum^k_{i=1}(\lambda^i_n)^{(p-N)/p}v_i (\cdot/\lambda^i_n)\|\rightarrow 0, \quad n\rightarrow \infty,$$

$$\|u_n\|^p\rightarrow \sum^k_{i=0}\|v_i\|^p, \quad n\rightarrow \infty,$$
$$\phi(v_0)+\sum^k_{i=1}\tilde\phi_\infty (v_i) =c,$$
where for $i\geq1$ $v_i$ is either identically zero, or for some $\lambda>0$ and up to the sign, it holds that

\begin{equation*}
v_i \equiv   \Big[\frac{\lambda^{\frac{1}{p-1}}N^{\frac{1}{p}}(\frac{N-p}{p-1})^{\frac{p-1}{p}}} {\lambda^{\frac{p}{p-1}}+|\cdot |^{\frac{p}{p-1} }}\Big]^{\frac{N-p}{p}} ,
\end{equation*}
and $$\tilde\phi_{\infty}(v_i)=\frac{S^{N/p}}{N} .$$ Moreover if $a\equiv 0,$ then either  $v_0$ is identically zero, or it holds that for some $\lambda>0$ and up to the sign

\begin{equation*}
v_0 \equiv   \Big[\frac{\lambda^{\frac{1}{p-1}}N^{\frac{1}{p}}(\frac{N-p}{p-1})^{\frac{p-1}{p}}} {\lambda^{\frac{p}{p-1}}+|\cdot |^{\frac{p}{p-1} }}\Big]^{\frac{N-p}{p}},
\end{equation*}
and $$\phi(v_0)=\tilde\phi_{\infty}(v_0)=\frac{S^{N/p}}{N} .$$}


\end{prop}
\begin{proof}
It follows by Theorem 5.4 of \cite{MW} and Proposition \ref{nonexxx}. When $a\equiv 0,$ $c$ can be computed explicitly by \cite{T}, using Proposition \ref{nonexxx}.\end{proof}

\section{Finite Morse index solutions}\label{morsection}
\subsection{Bounds on the number of nodal regions and proof of Theorem \ref{morse theorem}}
Let $\mathcal {O}$ be a domain of $\mathbb R^N,$ and 
$u\in W^{1,p}_{\mathrm{loc}}(\mathcal{O})$ be such that
 \begin{equation*}
-\Delta_p u =   |u|^{p^*-2}u \quad \textrm{in}\quad \mathcal D'(\mathcal {O}).
\end{equation*}  
For all $v\in C^1_c(\mathcal O)$ define $$\phi''_\infty(u)[v,v]=\int_{\R^N}|\nabla u|^{p-2}|\nabla v|^2+(p-2)|\nabla u|^{p-4}(\nabla u,\nabla v)^2-(p^*-1)|u|^{p^*-2}|v|^2.$$
We say that $u$ has Morse index $i(u)$, see for instance \cite{DFSV, F}, if $i(u)$ is the maximal dimension of the subspaces $V$ of  $C^1_c(\mathcal O)$ such that $$\phi''_\infty(u)[v,v]<0, \qquad \mathrm{for\, all}\,\,v\in V\setminus\{0\}.$$
\begin{proof}[Proof of Theorem \ref{morse theorem}]
Define on $\mathcal D_0^{1,p}(\mathcal O)$

$$\phi_\infty(u)=\int_{\mathcal O}\Big(\frac{|\nabla u|^{p}}{p}-\frac{|u|^{p^*}}{p^*}\Big)\textrm{d}x.
$$
The linearised functional is $$\phi''_\infty(u)[v,v]=\int_{\mathcal O}|\nabla u|^{p-2}|\nabla v|^2+(p-2)|\nabla u|^{p-4}(\nabla u,\nabla v)^2-(p^*-1)|u|^{p^*-2}|v|^2.$$
Pick the sequence $(u_n)$ as given by Lemma \ref{approx} and with $u|_\omega$ being the restriction of $u$ to $\omega.$ Extend by zero outside $\omega$ $u|\omega$ and all $u_n.$ By Lemma \ref{approx} we have
$$\phi''_\infty(u)[u_n,u_n]\longrightarrow \int_{\omega}|\nabla u|^{p}+(p-2)|\nabla u|^{p}-(p^*-1)|u|^{p^*}=\int_\omega(p-1)|\nabla u|^p-(p^*-1)|u|^{p*}.$$
This and Lemma \ref{prep} yields
$$\phi''_\infty(u)[u_n,u_n]\longrightarrow(p-p^*)\int_\omega|u|^{p*}<0.$$
This means that for every nodal domain $\omega$
there exists a direction $u_n\in C^{0,1}_c(\omega)$ (and by density in $C^1_c(\omega)$) such that $$\phi''_\infty(u)[u_n,u_n]<0$$
and this concludes the proof.
\end{proof}

\subsection{Proof of Theorem \ref{Halfspace0} }

\begin{proof}[Proof of Theorem \ref{Halfspace0}]
1) If $u\equiv 0$ then $i(u)=0.$\\
Assume that $ u \not \equiv 0$ and $i(u)\leq 1.$ By Theorem \ref{morse theorem} there is exactly one nodal region, say $A.$ If $A=\R_+^N$ we have a contradiction by \cite{MW}. If $A$ is a proper subset, we can assume, up to consider $-u$ instead of $u$,  that $A = \{ u>0\}$.  
Since $u\in C^{1}(\R_+^N)$ by Lemma \ref{lem} it follows that $\nabla u= \bold {0}$ on $\partial A.$ Pick now any interior point $p\in A.$ There exists a ball $B(p,R)\subset A$ centered at $p$ for some radius $R>0$ such that $B(p,R)$ touches internally $\partial A$ at some points. Let $p'$ be such an intersection point. Since $p'$ is a boundary point and satisfies the interior sphere condition, Hopf's boundary point lemma implies that $\nabla u (p')\neq \bold{0},$ which is a contradiction. This concludes the proof of part 1).\\
2) Let now $i(u)\leq 1.$ By Theorem \ref{morse theorem} $u$ has at most a nodal domain $A.$
If $A$ is a proper subset, the preceding proof of part 1) shows that $u\equiv 0.$ Otherwise if $A\equiv \R^N$ then the conclusion follows by  the recent classification result \cite{S}. And this concludes the proof of part 2).
\end{proof}

\begin{remark}\label{Rk2} The above proof shows also that $u$ cannot have a nodal domain $A$ surrounded by a region where $u$ is identically zero.  Moreover, all nodal domains have always some boundary points satisfying an interior sphere condition. \\We observe that in the case of (\ref{hh}) an alternative way to conclude the proof is by the strong maximum principle \cite{V}.
\end{remark}

\subsection{Starshaped domains: Proof of Theorem \ref{bddstar}}

\begin{proof} [Proof of Theorem \ref{bddstar} ]
By a refinement of Moser's iteration, see e.g. Appendix E of \cite{Peral} and \cite{Tr}, $u\in L^{\infty}(\Omega).$ By classical regularity results of DiBenedetto \cite{Di} and Liebermann \cite{Li}, we have that $u\in C^{1,\alpha}(\overline{\Omega}).$ By Theorem 1.1 of \cite{GV} it holds that the normal derivative $u_\nu=0$ at some point $x_0\in \partial \Omega.$ By Theorem \ref{morse theorem} $u$ has at most one nodal region. If $u$ were nontrivial, this would be in contrast with Hopf's boundary point lemma, see \cite{V}. It follows that $u\equiv0,$ and this concludes the proof.
 \end{proof}





\end{document}